\documentclass[12pt,english,reqno]{amsart}

\usepackage[right=2.5cm, left=2.5cm, top=2.5cm, bottom=2.5cm]{geometry}
\usepackage{hyperref}
\usepackage[all]{xy}

\numberwithin{equation}{section}

{\theoremstyle{plain}
\newtheorem{theorem}{Theorem}[section]
\newtheorem{lemma}[theorem]{Lemma}}

{\theoremstyle{remark}   
\newtheorem{remark}[theorem]{Remark}}

\title{On a variant of the Beckmann--Black problem}

\author{Fran\c{c}ois Legrand}

\email{francois.legrand@unicaen.fr}

\address{Normandie Univ., UNICAEN, CNRS, Laboratoire de Math\'ematiques Nicolas Oresme, 14000 Caen, France}

\begin{document}

\maketitle

\vspace{-10mm}

\begin{abstract}
Given a field $k$ and a finite group $G$, the Beckmann--Black problem asks whether every Galois field extension $F/k$ with group $G$ is the specialization at some $t_0 \in k$ of some Galois field extension $E/k(T)$ with group $G$ and $E \cap \overline{k} = k$. We show that the answer is positive for arbitrary $k$ and $G$, if one waives the requirement that $E/k(T)$ is normal. In fact, our result holds if ${\rm{Gal}}(F/k)$ is any given subgroup $H$ of $G$ and, in the special case $H=G$, we provide a similar conclusion even if $F/k$ is not normal. We next derive that, given a division ring $H$ and an automorphism $\sigma$ of $H$ of finite order, all finite groups occur as automorphism groups over the skew field of fractions $H(T, \sigma)$ of the twisted polynomial ring $H[T, \sigma]$.
\end{abstract}

\section{Introduction} \label{sec:intro}

The inverse Galois problem (over $\mathbb{Q}$), a question going back to Hilbert and Noether, asks whe\-ther every finite group is the Galois group of a Galois field extension of $\mathbb{Q}$. Various techniques, including cohomological methods, the study of Galois representations of the absolute Galois group of $\mathbb{Q}$, and the construction of Galois covers of $\mathbb{P}^1_\mathbb{Q}$ with specified Galois groups and the use of Hilbert's irreducibility theorem, allow to realize various finite groups $G$ as Galois groups over $\mathbb{Q}$, such as $G$ solvable (Shafarevich's theorem; see \cite[Theorem 9.6.1]{NSW08}), $G= {\rm{PSL}}_2(\mathbb{F}_p)$ where $p$ is an arbitrary prime number (Zywina \cite{Zyw15}), $G=S_n$ and $G=A_n$ (see, e.g., \cite[\S16.4]{FJ08}), $G$ sporadic but different from the Mathieu group M$_{23}$ (see, e.g., \cite[Chapter II, \S9]{MM18} for references), etc. Yet, the inverse Galois problem is still open.

In \cite{FK78}, using a result of Frucht \cite{Fru39} on the existence of finite undirected graphs ha\-ving neither loops nor isolated points and with pres\-cribed automorphism groups, E. Fried and Kollar prove the following for $k=\mathbb{Q}$: given a finite group $G$, there is a finite separable (non-necessarily normal) field extension $L/k$ with ${\rm{Aut}}(L/k)=G$ ($*$). Clearly, this unconditional conclusion is necessary for a po\-sitive ans\-wer to the inverse Galois problem, hence the interest in the result. Other proofs and/or generalizations are given in subsequent works by a number of authors. For example, in \cite{Fri80}, M. Fried uses Hilbert's irreducibility the\-o\-rem to reobtain ($*$) for $k=\mathbb{Q}$, and actually shows that, given a finite group $G$, there are infinitely number fields $L$ with ${\rm{Aut}}(L/\mathbb{Q})=G$. Independently, Takahashi \cite{Tak80} shows that ($*$) holds if $k$ is an arbitrary global field. Moreover, in \cite{Gey83}, Geyer reproves ($*$) for $k=\mathbb{Q}$; his method is similar to that of M. Fried but is simpler as Hilbert's irreducibility theorem is not used. Finally, in \cite{LP18}, Paran and the author show that, given a finite group $G$ and a Hilbertian field $k$ \footnote{Recall that a field $k$ is {\it{Hilbertian}} if Hilbert's irreducibility theorem holds over $k$. For example, global fields and rational function fields are Hilbertian. See, e.g., \cite{FJ08} for more on Hilbertian fields.}, there are infinitely many finite separable field extensions $L/k$ with ${\rm{Aut}}(L/k)=G$.

M. Fried's method is geometric and, although it is not explicitly stated there, it actually yields that, given a finite group $G$ and a field $k$ of characteristic zero, there is a finite separable field ex\-ten\-sion $E/k(T)$ with ${\rm{Aut}}(E/k(T))=G$ and $E \cap \overline{k} = k$. This was later extended to all fields by Deschamps and the author in \cite[th\'eor\`eme A]{DL21}. On the one hand, this result generalizes those from the last paragraph. On the other hand, it is necessary for a positive answer to the regular inverse Galois problem (the geometric approach to the inverse Galois problem) over an arbitrary field $k$, which asks whether every finite group is the Galois group of a Galois field extension $E/k(T)$ with $E \cap \overline{k}=k$. Recall that, although no counter-example is known and apart from an example of Koenigsmann \cite{Koe04}, all fields which are currently known to fulfill the regular inverse Galois problem contain an ample field \cite{Pop96}\footnote{Recall that a field $k$ is {\it{ample}} (or {\it{large}}) if every smooth geometrically irreducible $k$-curve has zero or infinitely many $k$-rational points. Ample fields include algebraically closed fields, the complete valued fields $\mathbb{Q}_p$, $\mathbb{R}$, $\kappa((Y))$, the field $\mathbb{Q}^{\rm{tr}}$ of totally real numbers, etc. See, e.g., \cite{Jar11, BSF13, Pop14} for more details.}.

Let us also mention that there is a number of other results in the literature on realizing finite groups as automorphism groups in the geometric function field case. For example, given a finite group $G$ and an algebraically closed field $k$, there are infinitely many non-isomorphic function fields $L$ in one variable over $k$ with ${\rm{Aut}}(L/k)=G$ (Madden--Valentini \cite{MV83}). In fact, given a non-trivial finite group $G$ and a function field $K$ in one variable over an arbitrary algebraically closed field $k$, there are infinitely many Galois field extensions $L/K$ with ${\rm{Gal}}(L/K) = {\rm{Aut}}(L/k) = G$, as proved by Greenberg \cite{Gre74} if $k=\mathbb{C}$, Stichtenoth if $K$ has genus at least 2 \cite{Sti84}, and Madan--Rosen in general \cite{MR92}.

\subsection{Function field extensions with specified specializations} \label{ssec:intro_1}

A generalization of \cite[th\'eor\`eme A]{DL21}, in the direction of finite embedding problems, is studied by Fehm, Paran, and the author in \cite{FLP19}. In \S\ref{sec:proof} of the present paper, we also go further than producing finite separable field extensions $E/k(T)$ with $E \cap \overline{k}=k$ and specified automorphism groups, but in another direction: we construct such extensions with specified specializations. Given $t_0 \in k$, by {\it{a specialization}} of $E/k(T)$ at $t_0$, we mean an extension of the form $(B/\frak{P})/k$, with $\frak{P}$ a prime ideal of the integral closure $B$ of $k[T]$ in $E$ containing $T - t_0$. See \S\ref{ssec:prelim_2} for more details.

\begin{theorem} \label{thm:intro_1}
Let $k$ be a field, $H \leq G$ finite groups, and $F/k$ a finite separable field extension with automorphism group $H$.

\vspace{0.5mm}

\noindent
{\rm{(1)}} If $F/k$ is Galois, there is a finite separable field extension $E/k(T)$ with ${\rm{Aut}}(E/k(T))=G$ and $E \cap \overline{k} = k$, in which $\langle T \rangle$ is unramified and for which every specialization at 0 equals $F/k$.

\vspace{0.5mm}

\noindent
{\rm{(2)}} If $H=G$, there is a finite separable field extension $E/k(T)$ with ${\rm{Aut}}(E/k(T))=G$ and $E \cap \overline{k} = k$, and for which some specialization at 0 equals $F/k$ (and the corresponding prime ideal is unramified in $E/k(T)$).
\end{theorem}

Our result relates to the Beckmann--Black problem, whose motivation was to explore the possible limitations of the geometric approach to the inverse Galois problem: for a field $k$ and a finite group $G$, is any given Galois field extension $F/k$ of group $G$ the specialization at some $t_0 \in k$ of some Galois field extension $E/k(T)$ with group $G$ and $E \cap \overline{k} = k$? As shown by Moret-Bailly \cite{MB01}, extending a former result in characteristic zero of Colliot-Th\'el\`ene \cite{CT00}, the answer is {\it{Yes}} for arbitrary $G$ if $k$ is ample. If $k=\mathbb{Q}$, the answer is known to be {\it{Yes}} for only a few groups $G$, including abelian groups (Beckmann \cite{Bec94}), some dihedral groups (Black \cite{Bla98, Bla99}), symmetric groups, and alternating groups (Mestre \cite{Mes90}, Kl\"uners--Malle \cite{KM01}), and no counter-example is known. While this gives support for the geometric approach, it should also be recalled that, if the answer to the Beckmann--Black problem is affirmative for every finite group and every field, then all fields fulfill the regular inverse Galois problem (D\`ebes \cite{Deb99a}). This last result is another motivation for the study of the Beckmann--Black problem, as it shows that positive results about this problem provide evidence for a positive answer to the (regular) inverse Galois problem, but it also shows that answering the Beckmann--Black problem positively in full generality is currently out of reach.

Given a field $k$ and a finite group $G$, taking $H=G$ in Theorem \ref{thm:intro_1}(1) yields immediately that the answer to the Beckmann--Black problem is affirmative, if one waives the requirement that $E/k(T)$ is normal\footnote{As the extensions $E/k(T)$ from Theorem \ref{thm:intro_1}(1) are not necessarily normal, there might be several prime ideals containing $T$ in such $E/k(T)$. However, as the theorem shows, the set of corresponding specializations is a singleton, and one may then speak of ``the" specialization of such $E/k(T)$ at 0.}. Once again, such an unconditional conclusion is necessary for a positive answer to the original question, hence the interest in the result. As to Theorem \ref{thm:intro_1}(2), it solves a variant for automorphism groups of the Beckmann--Black problem.

\subsection{Application to division rings} \label{ssec:intro_2}

The inverse Galois problem and its generalizations are traditionally studied over fields. Yet, thanks to the Artinian viewpoint, Galois theory of fields can be generalized to division rings (see \cite{Jac64, Coh95} for more details), thus allowing to study inverse Galois theory over division rings. See \cite{DL20, ALP20, Beh21, BDL20f, Leg20, FL20, Des21a} for recent works on (the non-commutative aspects of) this topic. 

In \S\ref{sec:proof_2}, we use Theorem \ref{thm:intro_1}(1) to give a non-Galois variant of a recent result of Behajaina \cite{Beh21} on solving the inverse Galois problem over division rings of the form $H(T, \sigma)$. Given a division ring $H$ and an automorphism $\sigma$ of $H$, by $H(T, \sigma)$, we mean the unique division ring containing the twisted polynomial ring $H[T, \sigma]$ and every element of which can be written as $rs^{-1}$ with $r \in H[T, \sigma]$ and $s \in H[T, \sigma] \setminus \{0\}$. We refer to \S\ref{ssec:prelim_1} for more details and only say here that, if $H$ is a field and $\sigma = {\rm{id}}_H$, then $H(T, \sigma)$ is the rational function field $H(T)$. 

Letting $h$ denote the center of $H$ and assuming that $\sigma$ has finite order, Behajaina shows that, if a given finite group $G$ is the Galois group of a Galois field extension of $h^{\langle \sigma \rangle}(T)$ which is totally split at $\langle T \rangle$, where $h^{\langle \sigma \rangle}$ is the fixed field of $\sigma$ in $h$, then $G$ is the Galois group of a Galois extension of $H(T, \sigma)$ (this is recalled as Lemma \ref{lem:angelot}). In particular, using that the assumption in this last implication is satisfied over ample fields (see \cite{Pop96, HJ98b}), the answer to the inverse Galois problem over $H(T, \sigma)$ is positive, if $h^{\langle \sigma \rangle}$ contains an ample field.

Here we show that, if a given finite group $G$ is the automorphism group of a finite separable field extension of $h^{\langle \sigma \rangle}(T)$ which is totally split at $\langle T \rangle$, then $G$ is an automorphism group over $H(T, \sigma)$ (see Lemma \ref{angelot}). By taking $F=k$ in Theorem \ref{thm:intro_1}(1), we then get the next result, which has no assumption on $h^{\langle \sigma \rangle}$:

\begin{theorem} \label{thm:intro_2}
Let $H$ be a division ring, $\sigma$ an automorphism of $H$ of finite order, and $G$ a finite group. There exists an extension of $H(T, \sigma)$ with automorphism group $G$.
\end{theorem}

{\bf{Acknowledgements.}} We thank Angelot Behajaina for helpful dis\-cussions about Lemmas \ref{lem:angelot} and \ref{angelot}, and the anonymous referees for several valuable comments and suggestions. The present work fits into Project TIGANOCO, which is funded by the European Union within the framework of the Operational Programme ERDF/ESF 2014-2020.

\section{Preliminaries on division rings and function field extensions} \label{sec:prelim}

\subsection{Division rings} \label{ssec:prelim_1}

A {\it{division ring}} is a non-necessarily commutative ring in which every non-zero element is invertible. Of course, commutative division rings are nothing but {\it{fields}}.

Given division rings $H \subseteq L$, we may consider $L$ as a left vector space over $H$ or a right vector space over $H$. In the sequel, we will always consider $L$ as a left vector space over $H$. We then say that $L/H$ is {\it{finite}} if the (left) dimension $[L:H]$ of $L$ over $H$ is finite.

Given division rings $H \subseteq L$, let ${\rm{Aut}}(L/H)$ be the {\it{automorphism group}} of $L/H$, i.e., the group of all automorphisms of $L$ fixing $H$ pointwise. We say that $L/H$ is {\it{outer}} if the only inner automorphism of $L$ lying in ${\rm{Aut}}(L/H)$ is the identity ${\rm{id}}_L$ of $L$. Clearly, if $L$ is a field, then $H$ is a field and $L/H$ is outer. Conversely, if $H$ is a field and $L/H$ is outer, then $L$ is a field (see, e.g., \cite[lemme 2.1]{BDL20f}). 

Following Artin, we say that an extension $L/H$ of division rings is {\it{Galois}} if every element $x$ of $L$ fulfilling $\sigma(x) =x$ for every $\sigma \in {\rm{Aut}}(L/H)$ is in $H$. If $L/H$ is Galois, then ${\rm{Aut}}(L/H)$ is the {\it{Galois group}} ${\rm{Gal}}(L/H)$ of $L/H$. Let $L/H$ be a Galois extension with finite Galois group. Then $L/H$ is finite, and $[L:H] = |{\rm{Gal}}(L/H)|$ if and only if $L/H$ is outer (see \cite[\S2, th\'eor\`eme]{Des18}). Moreover, if $H$ is of finite dimension over its center, then $L/H$ is outer (see, e.g., the first paragraph of the proof of \cite[corollaire 2]{DL20}). In particular, if $H$ is a field, then the field extension $L/H$ is finite, normal, and separable.

A (non-necessarily commutative) non-zero ring $R$ with no zero divisor is a {\it{right Ore domain}} if, for all $x, y \in R \setminus \{0\}$, there are $r, s \in R$ with $xr = ys \not=0$. If $R$ is a right Ore domain, there is a division ring $H$ which contains $R$ and every element of which can be written as $rs^{-1}$ with $r \in R$ and $s \in R \setminus \{0\}$ (see, e.g., \cite[Theo\-rem 6.8]{GW04}). Moreover, such a division ring $H$ is unique up to isomorphism (see, e.g., \cite[Proposition 1.3.4]{Coh95}).

Let $H$ be a division ring and $\sigma$ an automorphism of $H$. The {\it{twisted polynomial ring}} $H[T, \sigma]$ is the ring of  polynomials $a_0 + a_1 T + \cdots + a_m T^m$ with $m \geq 0$ and $a_0, \dots, a_m \in H$, whose addition is defined componentwise and multiplication fulfills $Ta = \sigma(a) T$ ($a \in H$). Note that $H[T, \sigma]$ is commutative if and only if $H$ is a field and $\sigma={\rm{id}}_H$. In the sense of Ore (see \cite{Ore33}), $H[T, \sigma]$ is the twisted polynomial ring $H[T, \sigma, \delta]$ in the variable $T$, where the $\sigma$-derivation $\delta$ is 0. The ring $H[T, \sigma]$ has no zero divisor, as the degree is additive on products, and is a right Ore domain (see, e.g., \cite[Theorem 2.6 and Corollary 6.7]{GW04}). The unique division ring which contains $H[T, \sigma]$ and each element of which can be written as $rs^{-1}$ with $r \in H[T, \sigma]$ and $s \in H[T, \sigma] \setminus \{0\}$ is then denoted by $H(T, \sigma)$. If $\sigma={\rm{id}}_H$, we write $H[T]$ and $H(T)$ instead of $H[T, {\rm{id}}_H]$ and $H(T, {\rm{id}}_H)$, respectively. If $H$ is a field, $H(T)$ is nothing but the usual field of fractions of the commutative polynomial ring $H[T]$.

The proof of Theorem \ref{thm:intro_2} will require the next lemma, which is almost contained in \cite{Beh21}. As announced in \S\ref{ssec:intro_2}, the lemma will be extended to automorphism groups in Lemma \ref{angelot}.

\begin{lemma} \label{lem:angelot}
Let $H$ be a division ring with center $h$, let $\sigma$ be an automorphism of $H$ of finite order $n$, let $h^{\langle \sigma \rangle} = \{x \in h \, : \, \sigma(x) = x\}$, and let $e$ be a Galois extension of $h^{\langle \sigma \rangle}(T^n)$ with finite Galois group and such that $e \subseteq h^{\langle \sigma \rangle}((T^n))$. Then the following three conclusions hold:

\vspace{0.5mm}

\noindent
{\rm{(1)}} $E = H(T, \sigma) \otimes_{h^{\langle \sigma \rangle}(T^n)} e$ is a division ring, which is a Galois extension of $H(T, \sigma)$,

\noindent
{\rm{(2)}} every element of ${\rm{Gal}}(E/H(T, \sigma))$ induces by restriction an element of ${\rm{Gal}}(e/h^{\langle \sigma \rangle}(T^n))$, and the corresponding map ${\rm{Gal}}(E/H(T, \sigma)) \rightarrow {\rm{Gal}}(e/h^{\langle \sigma \rangle}(T^n))$ is an isomorphism,

\vspace{0.5mm}

\noindent
{\rm{(3)}} $E/H(T, \sigma)$ is outer.
\end{lemma}

\begin{proof}[Comments on proof]
Conclusions (1) and (2) are a combination of \cite[lemme 2.1.1 \& proposition 2.1.2]{Beh21}. As to (3), the extension $E/H(T, \sigma)$ is Galois with finite Galois group and, by (2) and the definition of $E$, we have $|{\rm{Gal}}(E/H(T, \sigma))| = [E:H(T, \sigma)].$ Hence, as recalled above, $E/H(T, \sigma)$ is outer.
\end{proof}

\subsection{Function field extensions} \label{ssec:prelim_2}

Let $E/k(T)$ be a finite separable field extension. We say that $E/k(T)$ is {\it{$k$-regular}} if $E \cap \overline{k} = k$. To avoid confusion, let us recall that, if $B$ denotes the integral closure of $k[T]$ in $E$, then the prime ideal $\langle T - t_0 \rangle$ ($t_0 \in k$) of $k[T]$ is {\it{unramified}} in $E/k(T)$ if no prime ideal of $B$ containing $T- t_0$ ramifies in $E/k(T)$.

Given $t_0 \in k$, let $\frak{P}_1, \dots, \frak{P}_s$ be the prime ideals of $B$ containing $T-t_0$. For $i=1, \dots, s$, the finite extension $(B/\frak{P}_i)/k$ is {\it{a speciali\-zation}} of $E/k(T)$ at $t_0$, denoted by $E_{t_0, i}/k$.

First, assume that $E/k(T)$ is Galois, and fix $i \in \{1, \dots, s\}$. Then $E_{t_0, i}/k$ is normal. Moreover, let $D_{\frak{P}_i}$ (resp., $I_{\frak{P}_i})$ be the decomposition group (resp., the inertia group) of $E/k(T)$ at $\frak{P}_i$, and denote the reduction modulo $\frak{P}_i$ of any given element $x$ of $B$ by $\overline{x}$. Then the map

\begin{equation} \label{eq:psi}
\psi : \left \{ \begin{array} {ccc}
D_{\frak{P}_i} & \longrightarrow & {\rm{Aut}}({E}_{t_0, i}/k) \\
\sigma & \longmapsto & \overline{\sigma}
\end{array} \right. ,
\end{equation}
where $\overline{\sigma}(\overline{x}) = \overline{\sigma(x)}$ ($\sigma \in D_{\frak{P}_i}$, $x \in {B}$), is an epimorphism of kernel $I_{\frak{P}_i}$. If $\langle T - t_0 \rangle$ is unramified in ${E}/k(T)$, then ${E}_{t_0, i}/k$ is Galois and $\psi : D_{\frak{P}_i} \rightarrow {\rm{Gal}}({E}_{t_0,i}/k)$ is an isomorphism. Furthermore, for $i, j \in \{1, \dots, s\}$, there is $\sigma \in {\rm{Gal}}(E/k(T))$ with $\frak{P}_j = \sigma(\frak{P}_i)$, which then induces a $k$-isomorphism $E_{t_0,i} \rightarrow E_{t_0,j}$ (and yields $D_{\frak{P}_j} = \sigma D_{\frak{P}_i} \sigma^{-1}$, $I_{\frak{P}_j} = \sigma I_{\frak{P}_i} \sigma^{-1}$). We may then speak of {\it{the specialization}} of $E/k(T)$ at $t_0$, which is denoted by $E_{t_0}/k$ for simplicity.

Now, let $\widehat{E}$ be the Galois closure of $E$ over $k(T)$. If $\langle T - t_0 \rangle$ is unramified in $\widehat{E}/k(T)$, then $E_{t_0, 1}/k, \dots, E_{t_0, s}/k$ are separable and the compositum of the Galois closures of $E_{t_0, 1}, \dots, E_{t_0, s}$ over $k$ equals $\widehat{E}_{t_0}$. Moreover, if $[\widehat{E}_{t_0} : k] = [\widehat{E} : k(T)]$, then $s=1$, in which case we simply write $E_{t_0}/k$, and we have ${\rm{Aut}}(E_{t_0} / k) = {\rm{Aut}}(E/k(T))$.

\section{Proof of Theorem \ref{thm:intro_1}} \label{sec:proof}

Our aim here is to prove the following two statements for an arbitrary field $k$.

\vspace{2mm}

\noindent
($*/k$) {\it{Let $H \leq G$ be finite groups and let $F/k$ be a Galois extension of group $H$. There is a $k$-regular extension $E/k(T)$ with ${\rm{Aut}}(E/k(T)) = G$, such that $\langle T  \rangle$ is unramified in $E/k(T)$, and such that every speciali\-zation of $E/k(T)$ at $0$ equals $F/k$.}}

\vspace{2mm}

\noindent
($**/k$) {\it{Let $G$ be a finite group and let $F/k$ be a finite separable field extension with ${\rm{Aut}}(F/k)=G$. There exists a $k$-regular extension $E/k(T)$ with ${\rm{Aut}}(E/k(T)) = G$, and which fulfills the following: there is a prime ideal of the integral closure of $k[T]$ in $E$ containing $T$ which is unramified in $E/k(T)$ and such that the corresponding specialization of $E/k(T)$ equals $F/k$.}}

\subsection{Reduction to the case where $k$ is infinite} \label{ssec:proof_1}

We start with the next lemma, which will allow us, in particular, to reduce to the proofs of {\rm{($*/k$)}} and {\rm{($**/k$)}} for $k$ infinite. Note that the lemma is classical if the extension $E/k(X,T)$ below is assumed to be Galois.

\begin{lemma} \label{lemma:0}
Let $k$ be a field, $G$ a finite group, and $F/k$ a finite separable field extension. Assume there exists a finite separable field extension $E/k(X,T)$ with the following properties:

\vspace{0.5mm}

\noindent
{\rm{(1)}} $E \cap \overline{k} = k$,

\vspace{0.5mm}

\noindent
{\rm{(2)}} ${\rm{Aut}}(E/k(X,T)) = G$,

\vspace{0.5mm}

\noindent
{\rm{(3)}} $\langle T \rangle$ is unramified in $E/k(X)(T)$.

\vspace{0.5mm}

\noindent
Assume further that either one of the following two conditions holds:

\vspace{0.5mm}

\noindent
{\rm{(4)}} the completion at every prime ideal containing $X$ of every specialization of $E/k(X)(T)$ at $T=0$ equals $F((X))/k((X))$,

\vspace{0.5mm}

\noindent
{\rm{(5)}} the completion at some prime ideal containing $X$ of some specialization of $E/k(X)(T)$ at $T=0$ equals $F((X))/k((X))$.

\vspace{0.5mm}

\noindent
Then there is a $k$-regular extension $\mathcal{E}/k(X)$ with ${\rm{Aut}}(\mathcal{E}/k(X)) = G$ and which fulfills the following two properties:

\vspace{0.5mm}

\noindent
{\rm{(a)}} the prime ideal $\langle X \rangle$ is unramified in $\mathcal{E}/k(X)$ and every specialization of $\mathcal{E}/k(X)$ at $X=0$ equals $F/k$ (if {\rm{(4)}} holds),

\vspace{0.5mm}

\noindent
{\rm{(b)}} there is an unramified prime ideal of the integral closure of $k[X]$ in $\mathcal{E}$ which contains $X$ and such that the corresponding specialization of $\mathcal{E}/k(X)$ at $X=0$ equals $F/k$ (if {\rm{(5)}} holds).

\vspace{0.5mm}

\noindent
Moreover, if there is an extension $E/k(X,T)$ as above which is additionally Galois, then $\mathcal{E}/k(X)$ may be chosen to be Galois.
\end{lemma}

\begin{proof}
We first consider both cases in parallel, and make a case distinction in the last paragraph of the proof.

First, by (3) and as $k(X)$ is infinite, there is a primitive element of $E$ over $k(X)(T)$ which is integral over $k(X)[T]$ and whose minimal polynomial $A(X,T,Y) \in k(X)[T][Y]$ is such that $A(X,0,Y)$ is separable (see, e.g., \cite[corollaire 1.5.16]{Deb09}). Moreover, by (1), the polynomial $A(X,T,Y)$ is irreducible over $\overline{k}(X,T)$. In particular, $A(X, X^NT, Y)$ and $A(X, X^N T^{-1}, Y)$ are irreducible over $\overline{k}(X,T)$, for $N \geq 1$. We also let $\widehat{E}$ be the Galois closure of $E$ over $k(X,T)$, set $L = \widehat{E} \cap \overline{k}$, and let $B(X,T,Y) \in L(X, T)[Y]$ be the minimal polynomial of a primitive element of $\widehat{E}$ over $L(X,T)$. Note that $B(X,X^N T,Y)$ and $B(X, X^N T^{-1}, Y)$ are irreducible over $\overline{k}(X,T)$, for $N \geq 1$. At this stage, choose $N \geq 1$ arbitrary.

Then use either \cite[Proposition 13.2.1]{FJ08} if $k$ is infinite or \cite[Theorem 13.4.2 and Proposition 16.11.1]{FJ08} if $k$ is finite to get the existence of $t_0(X) \in k(X)$ such that 
$$A(X, X^Nt_0(X), Y), A(X, X^N t_0(X)^{-1}, Y) \in k(X)[Y],$$ 
$$B(X,X^N t_0(X),Y), B(X, X^N t_0(X)^{-1}, Y) \in L(X)[Y]$$ 
are irreducible over $\overline{k}(X)$ and separable. Up to replacing $t_0(X)$ by $t_0(X)^{-1}$, we may and will assume that $t_0(X)$ is of non-negative $X$-adic valuation. We then set $t'_0(X) = X^N t_0(X)$. 

Now, let $M$ be the field generated over $L(X)$ by a root of $B(X, t'_0(X), Y)$. As this last polynomial is irreducible over $L(X)$, we have $[M:k(X)] = [\widehat{E}:k(X,T)]$ and, therefore, $M/k(X)$ equals the specialization $\widehat{E}_{t'_0(X)}/k(X)$ of $\widehat{E}/k(X)(T)$ at $T = t'_0(X)$ (in particular, $\widehat{E}_{t'_0(X)}/k(X)$ is Galois). Then, by \S\ref{ssec:prelim_2}, the specialized field ${E}_{t'_0(X)}$ is a well-defined finite separable extension of $k(X)$ and ${\rm{Aut}}(E_{t'_0(X)}/k(X)) = {\rm{Aut}}(E/k(X,T))$, that is, ${\rm{Aut}}(E_{t'_0(X)}/k(X)) = G$ by (2). Furthermore, $E_{t'_0(X)}$ contains a root $y'$ of $A(X, t'_0(X), Y)$. As this last polynomial is irreducible over $k(X)$, we have $E_{t'_0(X)} = k(X, y').$ Then combine this last equality and the irreducibility of $A(X, t'_0(X), Y)$ over $\overline{k}(X)$ to get $E_{t'_0(X)} \cap \overline{k} = k.$ Finally, if $E/k(X,T)$ is additionally assumed to be Galois, i.e., if $E = \widehat{E}$, then we have $\widehat{E}_{t'_0(X)} = E_{t'_0(X)}$ and, hence, the specialization $E_{t'_0(X)}/k(X)$ is Galois.

Finally, we prove (a) or (b), depending on whether (4) or (5) holds. In both cases, set $A(X,0,Y) = A_1(X,Y) \cdots A_s(X,Y)$, where $A_i(X,Y)$ is irreducible over $k(X)$ ($1 \leq i \leq s$), and, for $i=1, \dots, s$, let $y_i$ be a root of $A_i(X,Y)$. Then $k(X,y_1)/k(X), \dots, k(X,y_s)/k(X)$ are exactly the specializations of $E/k(X,T)$ at $T=0$. First, assume (4) holds. Choosing $N$ sufficiently large, Krasner's lemma (see, e.g., \cite[Proposition 12.3]{Jar91}), applied to the separable polynomial $A(X,0,Y)$, yields that, for every root $y'$ of $A(X, t'_0(X), Y)$, there is a root $y$ of $A(X, 0, Y)$ with $k((X))(y') = k((X))(y)$. As there is a root $y'$ of $A(X, t'_0(X), Y)$ with $E_{t'_0(X)} = k(X)(y')$, we get from (4) that the completion of $E_{t'_0(X)}$ at every prime ideal containing $X$ equals $F((X))$, as needed for (a). Now, assume (5) holds. As before, Krasner's lemma, applied to the separable polynomial $A(X,0,Y)$, yields that, for every root $y$ of $A(X, 0, Y)$, there is a root $y''$ of $A(X, t'_0(X), Y)$ with $k((X))(y'') = k((X))(y)$. By (5), there is a root $y$ of $A(X, 0, Y)$ with $k((X))(y) = F((X))$. Pick a root $y''$ of $A(X, t'_0(X), Y)$ with $k((X))(y'') = k((X))(y)$. Then the completion at some prime ideal containing $X$ of $k(X,y'')/k(X)$ equals $F((X))/k((X))$. Finally, as $k(X,y'')$ and $k(X,y') = E_{t'_0(X)}$ are $k(X)$-isomorphic, we get that $k(X,y'')/k(X)$ is $k$-regular and has automorphism group $G$, as needed for (b).
\end{proof}

\begin{lemma} \label{lemma;1}
If {\rm{($*/k$)}} and {\rm{($**/k$)}} hold for every infinite field $k$, then {\rm{($*/k$)}} and {\rm{($**/k$)}} hold for every field $k$.
\end{lemma}

\begin{proof}
First, as every finite field extension of a finite field is Galois, we have {\rm{($*/k$)}} $\Rightarrow$ {\rm{($**/k$)}} for every finite field $k$. Hence, to get the lemma, it suffices to show that, if {\rm{($*/k$)}} holds for all infinite fields $k$, then {\rm{($*/k$)}} holds for all fields $k$. To that end, let $k$ be a field, $H \leq G$ finite groups, and $F/k$ a Galois extension of group $H$. For an indeterminate $X$, the extension $F(X)/k(X)$ is Galois of group $H$. Hence, as $k(X)$ is infinite, we get from ($*/k(X)$) that there is a $k(X)$-regular extension $E/k(X)(T)$ with ${\rm{Aut}}(E/k(X,T))=G$, in which $\langle T \rangle$ is unramified, and every specialization at $T=0$ of which equals $F(X)/k(X)$. Then $E \cap \overline{k} = k$ and $F((X))/k((X))$ is the completion at every prime ideal containing $X$ of every specialization of $E/k(X)(T)$ at $T=0$. It then remains to apply Lemma \ref{lemma:0} to conclude.
\end{proof}

\subsection{Preliminary lemmas} \label{ssec:proof_2}

The proofs of {\rm{($*/k$)}} and {\rm{($**/k$)}} for $k$ infinite, which are given in \S\ref{ssec:proof_3}, require the next three lemmas. The first one is \cite[Proposition 2.3]{LP18}:

\begin{lemma} \label{lem:LP}
Given a field $k$ and $x \in k$, set $P_x(T,Y)=Y^3 + (T-x) Y + (T-x) \in k[T][Y].$ Then $P_x(T,Y)$ has Galois group $S_3$ over $k(T)$. Moreover, if $k_x$ is the field generated over $k(T)$ by any given root of $P_x(T,Y)$, then $k_{x_1} \not= k_{x_2}$ for $x_1 \not=x_2$ in $k$.
\end{lemma}

Our second lemma is more or less known to experts, and is a slight generalization of the positive answer to the Beckmann--Black problem for symmetric groups over arbitrary fields:

\begin{lemma} \label{lem:s_n}
Let $k$ be a field, $F/k$ a finite Galois field extension, and $k \subseteq L \subseteq F$ an intermediate field whose Galois closure over $k$ equals $F$. Then, given $n \geq [L:k]$, there is a $k$-regular Galois field extension $E/k(T)$ with Galois group $S_n$, in which $\langle T \rangle$ is unramified, and such that $F/k$ equals the specialization $E_0/k$ of $E/k(T)$ at 0.
\end{lemma}

\begin{proof}
First, assume $k$ is Hilbertian. Let $x$ be a primitive element of $L$ over $k$, and let $P(Y) \in k[Y]$ be the minimal polynomial of $x$ over $k$. Since $k$ is Hilbertian, it is, in particular, infinite, and we may then find $\alpha_1, \dots, \alpha_{n-[L:k]} \in k$ such that $P_0(Y) = (Y- \alpha_1) \cdots (Y- \alpha_{n - [L:k]}) P(Y)$ is separable. Using once more that $k$ is infinite, there are $\beta_1, \dots, \beta_n \in k$ such that $P_1(Y) = (Y - \beta_1) \cdots (Y- \beta_n)$ is separable. Finally, since $k$ is Hilbertian, there is a monic degree $n$ polynomial $Q(Y) \in k[Y]$ with Galois group $S_n$ over $k$; note that $Q(Y)$ is necessarily separable.

Pick $a$ in $k \setminus \{0,1\}$. Then, by polynomial interpolation, there is a monic degree $n$ polynomial $R(T, Y) \in k[T][Y]$ with $R(i, Y) = P_i(Y)$ for $i=0,1$, and with $R(a,Y) = Q(Y)$. Let $E$ be the splitting field over $k(T)$ of $R(T,Y)$. For $i=0,1$, since $R(i,Y)$ is separable, $\langle T - i \rangle$ is unramified in $E/k(T)$ and the specialized field $E_i$ equals the splitting field of $R(i,Y)$ over $k$. For the choice $i=0$, we get that $\langle T \rangle$ is unramified in $E/k(T)$, and that $E_0/k=F/k$. For the choice $i=1$, we get that $E$ embeds into $k((T))$; in particular, $E/k(T)$ is $k$-regular. Moreover, in the same way, the specialized field $E_a$ equals the splitting field of $R(a,Y)$ over $k$. Hence, $E_a/k$ has Galois group $S_n$, thus showing that $E/k(T)$ also has Galois group $S_n$.

We now prove the lemma. Given $n \geq [L:k]$, note that $F(X)/k(X)$ is finite Galois, that $F(X)$ is the Galois closure of $L(X)$ over $k(X)$, and that $n \geq [L(X):k(X)]$. As $k(X)$ is Hilbertian (see the proof of Lemma \ref{lemma:0} for references), the two paragraphs above yield a $k(X)$-regular Galois field extension $E/k(X)(T)$ of group $S_n$, in which $\langle T \rangle$ is unramified, and such that $F(X)/k(X)$ is the specialization of $E/k(X)(T)$ at $T=0$. In particular, $E \cap \overline{k}=k$ and the completion at every prime ideal containing $X$ of every specialization of $E/k(X)(T)$ at $T=0$ equals $F((X))/k((X))$. It then remains to apply Lemma \ref{lemma:0} to conclude. 
\end{proof}

Finally, we have the following elementary lemma about trinomials of degree 3:

\begin{lemma} \label{lem:trinomial}
Let $k$ be an infinite field. There exists $a \in k$ such that $Y^3 + a Y + a$ is separable and such that the splitting field over $k$ of this polynomial equals $k$.
\end{lemma}

\begin{proof}
We fix $\alpha \in k$ fulfilling these conditions: $\alpha \not \in \{0,1,-1, -2\}$, $\alpha \not= -1/2$ if the characteristic of $k$ does not equal 2, and $\alpha^2 + \alpha + 1 \not=0$; such $\alpha$ exists as $k$ is infinite. Then 
$$\beta = \frac{-\alpha^2 - \alpha - 1}{\alpha^2 + \alpha}$$
is a well-defined non-zero element of $k$ and all the roots of the polynomial
$$P(Y) = (Y - \beta)(Y - \beta \alpha)(Y - \beta(-1-\alpha)) \in k[Y]$$
are in $k$. Moreover, the roots of $P(Y)$ are pairwise distinct, and we have
$$P(Y)= Y^3 + \frac{(-\alpha^2 - \alpha - 1)^3}{(\alpha^2 + \alpha)^2} Y + \frac{(-\alpha^2 - \alpha - 1)^3}{(\alpha^2 + \alpha)^2},$$
thus ending the proof of the lemma.
\end{proof}

\subsection{Proofs of {\rm{($*/k$)}} and {\rm{($**/k$)}} for $k$ infinite} \label{ssec:proof_3}

To prove Theorem \ref{thm:intro_1}, it suffices, by Lemma \ref{lemma;1}, to prove {\rm{($*/k$)}} and {\rm{($**/k$)}} for $k$ infinite. We prove both statements in parallel. Let $k$ be an infinite field, $H \leq G$ finite groups, and $F/k$ a finite separable field extension of automorphism group $H$. Assume $F/k$ Galois or $H=G$.

\begin{lemma} \label{lemma:BB2}
Given an indeterminate $X$, there is a $k$-regular Galois extension $M/k(X)$ and intermediate fields $k(X) \subseteq L \subseteq E \subseteq M$ with the following properties:

\vspace{0.5mm}

\noindent
{\rm{(1)}} $\langle X \rangle$ is unramified in $M/k(X)$,

\vspace{0.5mm}

\noindent
{\rm{(2)}} ${\rm{Aut}}(E/L) = G$,

\vspace{0.5mm}

\noindent
{\rm{(3)}} if $F/k$ is Galois, then $E/k(X)$ is Galois and $F/k$ is the specialization of $E/k(X)$ at $0$, 

\vspace{0.5mm}

\noindent
{\rm{(4)}} if $H=G$, then $[E:L] = [F:k]$ and there is a prime ideal $\frak{P}$ of the integral closure of $k[X]$ in $E$ containing $X$ such that the residue field of $E$ at $\frak{P}$ equals $F$.
\end{lemma}

\begin{proof}
First, assume $F/k$ is Galois. Given $n \geq |G|$, let $M/k(X)$ be a $k$-regular Galois extension of group $S_n$, in which $\langle X \rangle$ is unramified, and such that the specialization $M_0/k$ of $M/k(X)$ at 0 equals $F/k$; such an extension exists by Lemma \ref{lem:s_n}. Then (2) and (3) hold with $E=M$ and $L$ equal to the fixed field of $G$ in $M$.

Now, suppose $H=G$. Let $\widehat{F}$ be the Galois closure of $F$ over $k$ and set $\widehat{H} = {\rm{Gal}}(\widehat{F}/k)$.  Given $n \geq |\widehat{H}|$, Lemma \ref{lem:s_n} yields a $k$-regular Galois extension $M/k(X)$ of group $S_n$, in which $\langle X \rangle$ is unramified, and with $M_0= \widehat{F}$. Let $\frak{Q}$ be a prime ideal of the integral closure of $k[X]$ in $M$ containing $X$, let $D_\frak{Q}$ be the decomposition group of $M/k(X)$ at $\frak{Q}$, let $\psi : D_\frak{Q} \rightarrow {\rm{Gal}}(\widehat{F}/k)$ be as in \eqref{eq:psi}, and let $L$ be the fixed field of $D_\frak{Q}$ in $M$. Note that the residue field of $L$ at the restriction of $\frak{Q}$ equals $k$. Setting $K = {\rm{Gal}}(\widehat{F} / F)$, let $E$ be the fixed field of $\psi^{-1}(K)$ in $M$. Then $[E:L] = [F:k]$, the residue field of $E$ at the restriction $\frak{P}$ of $\frak{Q}$ equals $F$, and ${\rm{Aut}}(E/L) = {\rm{Aut}}(F/k) =H$. As $H=G$, we get ${\rm{Aut}}(E/L) = G$, as needed.
\end{proof}

Fix a $k$-regular Galois extension $M/k(X)$ and intermediate fields $k(X) \subseteq L \subseteq E \subseteq M$ as in Lemma \ref{lemma:BB2}, and let $B$ be the integral closure of $k[X]$ in $E$. If $F/k$ is Galois, we arbitrarily fix a prime ideal $\frak{P}$ of $B$ containing $X$. Note that, by (1) and (3) in Lemma \ref{lemma:BB2}, the completion $E_\frak{P}$ of $E$ at $\frak{P}$ equals $F((X))$. If $H=G$, we fix a prime ideal $\frak{P}$ of $B$ containing $X$ as in Lemma \ref{lemma:BB2} and, by (1) and (4) in this last lemma, we also have $E_\frak{P} = F((X))$. Moreover, the infiniteness of $k$ and Lemma \ref{lem:trinomial} yield an element $a$ of $k$ such that $Y^3 + a Y + a$ is separable and such that the splitting field over $k$ of this last polynomial equals $k$ (in both cases).

Now, let $x_0$ be a primitive element of $L$ over $k(X)$, assumed to be integral over $k[X]$. By Krasner's lemma and as $L \subseteq F((X))$, there exists a positive integer $m$ such that the splitting 

\noindent
fields over $F((X))$ of the separable polynomials
$$Y^3 + (a- x_0 X^m)Y + (a-x_0 X^m) \, \, \, {\rm{and}} \, \, \, Y^3+aY+a$$
coincide. Set $x = x_0 X^m$. Then $x$ is a primitive element of $L$ over $k(X)$. Moreover, by the definition of $a$, the splitting field over $k$ of $Y^3+aY+a$ equals $k$. Hence, the splitting field over $F((X))$ of this last polynomial equals $F((X))$ and, consequently, all roots of $Y^3 + (a-x) Y + (a-x)$ are elements of $F((X))$.

Finally, let $T$ be an indeterminate, let $y$ be a root of 
$$P_x(T,Y) = Y^3 + (T-x) Y + (T-x) \in L[T][Y],$$
and let $\mathcal{E} = E(T,y)$. By Lemma \ref{lem:LP}, we have $[\mathcal{E} : E(T)] = 3 = [L(T,y) : L(T)]$. Moreover, letting $K$ denote the Galois closure of $E$ over $L$, the same lemma yields $[K(T)(y):K(T)]=3$. Hence, $K(T)$ and $L(T,y)$ are linearly disjoint over $L(T)$. As $K(T)$ is the Galois closure of $E(T)$ over $L(T)$, we get ${\rm{Aut}}(E(T,y) / L(T,y)) = {\rm{Aut}}(E(T)/ L(T)) = G$.

The following diagram of field extensions summarizes the construction:
\begin{figure}[h!]
{\tiny{\[ \xymatrix{
& & & & E(T,y) \ar@{-}[rd] \ar@{-} @/^1pc/[rd]^-{G} \ar@{=}[rr] & & \mathcal{E} \ar@{-}[dddd] \\
& & & E(T) \ar@{-}[rd] \ar@{-} @/^1pc/[rd]^{G} \ar@{-}[ru]^3 & & L(T,y) & \\
F \ar@{-}[dd] \ar@{-} @/_1pc/[dd]_{H} & & E \ar@{-}[ru] \ar@{-}[rd] \ar@{-}[dd] \ar@{-} @/^1pc/[rd]^{G} & & L(T) \ar@{-}[ru]^3 & & \\
& & \ar@{->}[ll]_{X=0 \, {\rm{(if}} \, F/k \, {\rm{is}} \, {\rm{Galois)}}}^{{\rm{mod}} \, \frak{P} \, {\rm{(if}} \, H=G{\rm{)}}} & L = k(X,x) \ar@{-}[ru] & & & \\
k & & k(X) \ar@{-}[ru] & & & &  k(X,T)
}
\]}}
\end{figure}

\begin{lemma} \label{aut_0}
{\rm{(a)}} We have ${\rm{Aut}}(\mathcal{E} / k(X,T)) = G$.

\vspace{0.5mm}

\noindent
{\rm{(b)}} We have $\mathcal{E} \cap \overline{k} = k$.
\end{lemma}

\begin{proof}
(a) The proof is similar to that of \cite[Lemma 4.5]{FLP19}.

Firstly, we have ${\rm{Aut}}(\mathcal{E}/k(X,T)) = {\rm{Aut}}(\mathcal{E}/L(T))$. Indeed, let $\sigma$ be any element of the automorphism group ${\rm{Aut}}(\mathcal{E}/k(X,T))$. If $\sigma \not \in {\rm{Aut}}(\mathcal{E}/L(T))$, then $\sigma(x) \not= x$ (as $L = k(X,x)$), and $\sigma(y)$ is a root of 
$$P_{\sigma(x)}(T, Y)=Y^3 + (T-\sigma(x))Y + (T-\sigma(x))  \in \widehat{E}[T][Y],$$
where $\widehat{E}$ is the Galois closure of $E$ over $k(X)$. Lemma \ref{lem:LP} then gives $\widehat{E}(T, \sigma(y)) \not=\widehat{E}(T,y)$. Since $\widehat{E}$ is the compositum of the $k(X)$-conjugates of $E$, we get $E(T, \sigma(y)) \not= E(T,y)$. As $\sigma(\mathcal{E}) \subseteq \mathcal{E}$, we get that $E(T, \sigma(y))$ is strictly contained in ${E}(T,y)$. But Lemma \ref{lem:LP} yields 
$$[E(T,y) : E(T)]=3= [\widehat{E}(T,\sigma(y)) : \widehat{E}(T)] \leq [{E}(T,\sigma(y)) : E(T)].$$

Secondly, ${\rm{Aut}}(\mathcal{E}/L(T)) = {\rm{Aut}}(\mathcal{E}/L(T,y))$ ($=G$). Indeed, given $\sigma \in {\rm{Aut}}(\mathcal{E}/L(T))$, we have to show that $\sigma$ fixes $y$. Assume $\sigma$ does not. Then $\sigma(y)$ is another root of $P_{x}(T,Y)$ and it is in $\mathcal{E}$. Hence, $\mathcal{E}$ contains all the roots of $P_{x}(T,Y)$ (as this last polynomial has degree 3 in $Y$). By Lemma \ref{lem:LP}, we get $[\mathcal{E}:E(T)] = 6$, a contradiction.

\vspace{1mm}

\noindent
(b) As $M/k(X)$ is as in Lemma \ref{lemma:BB2}, we have $[E\overline{k}:\overline{k}(X)] = [E:k(X)]$, and so $[E\overline{k}(T) :\overline{k}(X,T)] = [E(T):k(X,T)]$. Moreover, by Lemma \ref{lem:LP}, we have $[E\overline{k}(T,y) : E \overline{k}(T)] = 3 = [E(T,y) : E(T)]$. Hence, $[\mathcal{E} \overline{k} : \overline{k}(X,T)] = [\mathcal{E} : k(X,T)]$, thus ending the proof.
\end{proof}

To conclude the proof of Theorem \ref{thm:intro_1}, it then suffices, by Lemmas \ref{lemma:0} and \ref{aut_0} to show the following three statements, where $a$ is introduced after the proof of Lemma \ref{lemma:BB2}:

\vspace{1mm}

\noindent
{\rm{(1)}} $\langle T -a \rangle$ is unramified in $\mathcal{E}/k(X,T)$,

\vspace{1mm}

\noindent
{\rm{(2)}} $F((X))/k((X))$ is the completion at every prime ideal containing $X$ of every specialization of $\mathcal{E}/k(X)(T)$ at $T=a$ (if $F/k$ is Galois),

\vspace{1mm}

\noindent
{\rm{(3)}} $F((X))/k((X))$ is the completion at some prime ideal containing $X$ of some specialization of $\mathcal{E}/k(X)(T)$ at $T=a$ (if $H=G$).

\begin{proof}[Proof of {\rm{(1)}}]
Clearly, $\langle T-a \rangle$ is unramified in $E(T)/k(X,T)$. Moreover, from our choice of $a$ and $x$, the polynomial $P_x(a,Y)$ is separable. Hence, $\langle T-a \rangle$ is unramified in $\widetilde{\mathcal{E}} / E(T)$, where $\widetilde{\mathcal{E}}$ denotes the splitting field over $E(T)$ of $P_x(T,Y)$. Combining the two unramified conclusions yields that $\langle T-a \rangle$ is unramified in $\mathcal{E}/k(X,T)$.
\end{proof}

\begin{proof}[Proofs of {\rm{(2)}} and {\rm{(3)}}]
We prove both statements in parallel. Since $P_x(a,Y)$ is separable, all spe\-cia\-lizations $\widetilde{\mathcal{E}}_{a,i}/k(X)$ of $\widetilde{\mathcal{E}}/k(X)(T)$ at $T=a$ coincide and $\widetilde{\mathcal{E}}_{a,i}$ is the splitting field over $E$ of $P_x(a,Y)$ for every $i$. For simplicity, we write $\widetilde{\mathcal{E}}_a/k(X)$ for ``the" specialization of $\widetilde{\mathcal{E}}/k(X)(T)$ at $T=a$. Considering the prime ideal $\frak{P}$ introduced after the proof of Lemma \ref{lemma:BB2}, recall that the completion $E_\frak{P}$ of $E$ at $\frak{P}$ equals $F((X))$. It then suffices to show that the completion $(\widetilde{\mathcal{E}}_a)_\frak{P}$ of $\widetilde{\mathcal{E}}_a$ at $\frak{P}$ also equals $F((X))$. But the field $(\widetilde{\mathcal{E}}_a)_\frak{P}$ is the splitting field over $E_\frak{P} = F((X))$ of $P_x(a,Y)$ and, as seen after the proof of Lemma \ref{lemma:BB2}, all roots of this last polynomial are in $F((X))$.
\end{proof}

\section{Proof of Theorem \ref{thm:intro_2}} \label{sec:proof_2}

We need the following variant for automorphism groups of Lemma \ref{lem:angelot}:

\begin{lemma} \label{angelot}
Let $H$ be a division ring of center $h$, let $\sigma$ be an automorphism of $H$ of finite order $n$, let $h^{\langle \sigma \rangle} = \{x \in h \, : \, \sigma(x) = x\}$, and let $e$ be a finite separable field extension of $h^{\langle \sigma \rangle}(T^n)$ whose Galois closure embeds into $h^{\langle \sigma \rangle}((T^n))$. Then the next three conclusions hold:

\vspace{0.5mm}

\noindent
{\rm{(1)}} $H(T, \sigma) \otimes_{h^{\langle \sigma \rangle}(T^n)} e$ is a division ring, which is finite over $H(T, \sigma)$,

\noindent
{\rm{(2)}} ${\rm{Aut}}((H(T, \sigma) \otimes_{h^{\langle \sigma \rangle}(T^n)} e)/H(T, \sigma)) = {\rm{Aut}}(e/h^{\langle \sigma \rangle}(T^n))$,

\vspace{0.5mm}

\noindent
{\rm{(3)}} $(H(T, \sigma) \otimes_{h^{\langle \sigma \rangle}(T^n)} e)/H(T, \sigma)$ is outer.
\end{lemma}

\begin{proof}[Proof of Lemma \ref{angelot}]
Let $\widehat{e}$ be the Galois closure of $e$ over $h^{\langle \sigma \rangle}(T^n)$. As ${\rm{Gal}}(\widehat{e} / h^{\langle \sigma \rangle}(T^n))$ is finite and $\widehat{e} \subseteq h^{\langle \sigma \rangle}((T^n))$, Lemma \ref{lem:angelot} gives that $\widehat{E} = H(T, \sigma) \otimes_{h^{\langle \sigma \rangle}(T^n)} \widehat{e}$ is a Galois outer extension of $H(T, \sigma)$ and that the restriction map ${\rm{res}} : {\rm{Gal}}(\widehat{E} / H(T, \sigma)) \rightarrow {\rm{Gal}}(\widehat{e} / h^{\langle \sigma \rangle}(T^n))$ is a well-defined isomorphism. Set $K = {\rm{Gal}}(\widehat{e} / e)$ and let $E$ be the fixed division ring of ${\rm{res}}^{-1}(K)$ in $\widehat{E}$. As $H(T, \sigma) \subseteq E \subseteq \widehat{E}$ and $\widehat{E}/H(T, \sigma)$ is outer, $E/H(T, \sigma)$ is also outer. Moreover, as recalled in \S\ref{ssec:prelim_1}, $\widehat{E}/H(T, \sigma)$ is finite and, hence, $E/H(T, \sigma)$ is also finite. Furthermore, using that $\widehat{E}/H(T, \sigma)$ is Galois, finite, and outer, we get from, e.g., \cite[Theorem 3.3.11]{Coh95} that
$${\rm{Aut}}(E/H(T, \sigma)) = N_{{\rm{Gal}}(\widehat{E} / H(T, \sigma))}({\rm{res}}^{-1}(K)) / {\rm{res}}^{-1}(K),$$ where $N_{{\rm{Gal}}(\widehat{E} / H(T, \sigma))}({\rm{res}}^{-1}(K))$ denotes the normalizer of ${\rm{res}}^{-1}(K)$ in ${\rm{Gal}}(\widehat{E} / H(T, \sigma))$. But, via the isomorphism ${\rm{res}}$, we have 
$$N_{{\rm{Gal}}(\widehat{E} / H(T, \sigma))}({\rm{res}}^{-1}(K))/{\rm{res}}^{-1}(K) = N_{{\rm{Gal}}(\widehat{e} / h^{\langle \sigma \rangle}(T^n))}(K)/K,$$ 
where $N_{{\rm{Gal}}(\widehat{e} / h^{\langle \sigma \rangle}(T^n))}(K)$ denotes the normalizer of $K$ in ${\rm{Gal}}(\widehat{e} / h^{\langle \sigma \rangle}(T^n))$. Since the latter quotient group equals ${\rm{Aut}}(e/h^{\langle \sigma \rangle}(T^n))$, we get ${\rm{Aut}}(E/H(T, \sigma))={\rm{Aut}}(e/h^{\langle \sigma \rangle}(T^n))$.

To conclude the proof, it remains to show
\begin{equation} \label{eq:angelot}
H(T, \sigma) \otimes_{h^{\langle \sigma \rangle}(T^n)} e = E.
\end{equation}
To that end, note first that ${\rm{res}}({\rm{Gal}}(\widehat{E} / (H(T, \sigma) \otimes_{h^{\langle \sigma \rangle}(T^n)} e))) \subseteq K,$ i.e., ${\rm{Gal}}(\widehat{E} / (H(T, \sigma) \otimes_{h^{\langle \sigma \rangle}(T^n)} e)) \subseteq {\rm{res}}^{-1}(K)$. Then use  \cite[Theorem 3.3.11]{Coh95} once more to get that $\widehat{E}/E$ is Galois with Galois group ${\rm{res}}^{-1}(K)$. In particular, we have $E \subseteq H(T, \sigma) \otimes_{h^{\langle \sigma \rangle}(T^n)} e$. Moreover, since $\widehat{E} / H(T, \sigma)$ is outer, $\widehat{E}/E$ is also outer. As $\widehat{E}/E$ is Galois with finite Galois group ${\rm{res}}^{-1}(K)$, we get $[\widehat{E}:E] = |{\rm{res}}^{-1}(K)|$ (as recalled in \S\ref{ssec:prelim_1}), i.e., $[\widehat{E}:E] = [\widehat{e} : e]$. We then have
$$[E : H(T, \sigma)] = \frac{[\widehat{E} : H(T, \sigma)]}{[\widehat{E}:E]} = \frac{[\widehat{e}:h^{\langle \sigma \rangle}(T^n)]}{[\widehat{e}:e]} = [e : h^{\langle \sigma \rangle}(T^n)] = [H(T, \sigma) \otimes_{h^{\langle \sigma \rangle}(T^n)} e : H(T, \sigma)].$$
Hence, \eqref{eq:angelot} holds, as needed.
\end{proof}

\begin{proof}[Proof of Theorem \ref{thm:intro_2}]
Let $n$ be the order of $\sigma$. By Theorem \ref{thm:intro_1}(1), there is a finite separable field extension $e / h^{\langle \sigma \rangle}(T^n)$ of automorphism group $G$ which is totally split at $\langle T^n \rangle$. Then the Galois closure of $e$ over $h^{\langle \sigma \rangle}(T^n)$ embeds into $h^{\langle \sigma \rangle}((T^n))$ and it remains to apply Lemma \ref{angelot} to conclude the proof.
\end{proof}

\begin{remark}
Given a division ring $H$ with center $h$, recall that the center of $H(T)$ equals $h(T)$ (see, e.g., \cite[Proposition 2.1.5]{Coh95}). Hence, from the last two proofs (with $\sigma={\rm{id}}_H$), we have this conclusion: given a finite group $G$ and a division ring $H$ with center $h$, there is a finite outer extension $E$ of $H(T)$ with ${\rm{Aut}}(E/H(T))=G$ and whose center $e$ is an $h$-regular extension of $h(T)$. As recalled in \S\ref{ssec:prelim_1}, a division ring that is an outer extension of a field is necessarily a field. Hence, given a field $H$, we get that, given a finite group $G$, there is an $H$-regular field extension $E$ of $H(T)$ with ${\rm{Aut}}(E/H(T))=G$, which is \cite[th\'eor\`eme A]{DL21}.
\end{remark}

\bibliography{Biblio2}
\bibliographystyle{alpha}

\end{document}